\documentclass[11pt,reqno]{amsart}
\usepackage[bindingoffset=0.2in,
            left=1.25in,
            right=1.25in,
            top=1.25in,
            bottom=1.25in,
            footskip=.25in]{geometry}
\usepackage[greek, english]{babel}
\usepackage{amsmath}
\usepackage{amsfonts}
\usepackage{amsthm}
\usepackage{mathtools}
\usepackage{float}
\usepackage{hyperref}
\usepackage{xcolor}
\hypersetup{%
  colorlinks=true,%
  citecolor=blue,%
  linkcolor = black%
}  
\usepackage{caption}

\DeclareMathOperator{\nls}{null}

\newcounter{contentcounter}
\counterwithin{contentcounter}{section}

\newtheorem{theorem}[contentcounter]{Theorem}
\newtheorem{lemma}[contentcounter]{Lemma}
\newtheorem{definition}[contentcounter]{Definition}
\newtheorem{prop}[contentcounter]{Proposition}
\newtheorem*{pf}{Proof}

\newtheorem*{remark}{Remark}
\newtheorem*{example}{Example}
\usepackage{graphicx}

\begin{document}
\title[The Evans function and spectral distance]{The Evans function as a lower bound on the spectral distance function}

\author[G.\ Bayliss]{George Bayliss}
\address{University of Illinois, Urbana,
	IL 61801, USA}
\email{bayliss2@illinois.edu}

\author[J.\ Bronski]{Jared C. Bronski}
\address{Department of Mathematics, University of Illinois, Urbana, IL 61801, USA}
\email{bronski@illinois.edu}

\begin{abstract}

The Evans function is an analytic function that encodes information about the intersection of certain subspaces in ODE boundary value problems. As such it is a useful tool for computing the spectrum of boundary value problems arising in the stability of coherent structures. In typical applications one is interested in the roots of the Evans function, but the overall normalization is somewhat arbitrary.  
We present a natural normalization of the Evans function on compact domains such that the magnitude of the Evans function provides a lower bound on the distance to the nearest point in the spectrum. In other words the magnitude of the Evans function at a point in the resolvent set implies that a ball about the point in question lies in the resolvent set. Thus, when appropriately normalized, not only does the Evans function $E(\lambda)$ vanish if and only if $\lambda$ lies in the spectrum of the operator in question, but a non-zero value for the Evans function guarantees that a disk of radius $|E(\lambda^*)|$ about the point $\lambda^*$ lies in the resolvent set.

We present some calculations for some common sets of boundary conditions on a compact interval, and present some numerical experiments for 2nd and 4th order self-adjoint operators and for a linearized modified Korteweg–De Vries equation. 
\end{abstract}

\maketitle

\section{Introduction} \
The Evans function has become a commonplace tool in the study of the eigenvalue problems that arise in the study of nonlinear wave propagation\cite{B2009,OZ2003,S1992,HZ2006,AB2010,LMNT2009,BHLL2018,BNSVW2018}. Originally defined by Evans in his studies of propagation of impulses in models of nerve axons\cite{Evans1,Evans2,Evans3,Evans4} it has since been applied to many 
other types of models including cardiac tissue\cite{CO2004,BJSW2008,H2010,L2025}, water waves\cite{L2000,DP2006,B2013,hur2023unstable,HJ2015}, laser cavities\cite{KKS2004}, and elasticity theory\cite{LLM2011,IL2006,CI2019}, to name but a few. It has become an important theoretical and numerical tool for understanding spectral properties of boundary value problems.\cite{sandstede2002stability,CO2004,KKS2004,BJSW2008,MN2008,CB2014,BHLL2018,BNSVW2018,CLS2024}. Many investigations of the stability of coherent structures such as standing and traveling wave require the computation of an Evans function. 

At its heart the Evans function is an analytic object which detects the intersection of certain subspaces. Typically one constructs the Evans function by considering the subspaces of solutions to some governing ordinary differential equation which satisfy some prescribed boundary conditions at the lefthand and righthand boundaries. The subspaces of boundary conditions are then propagated to a common point; intersections of these subspaces represent a solution to the associated boundary value problem and a zero of the Evans function. 

From the nature of this definition the Evans function is usually treated projectively: the magnitude of the function is not important, only the question of whether the Evans function is zero or non-zero is of consequence. The Evans function detects something topological, the intersection or non-intersection of a particular pair of subspaces. The point of this paper is to point out, for many eigenvalue problems of the sort that arise in applications, that there is a natural scaling that gives a geometric significance to the magnitude of the Evans function. In particular for some eigenvalue problems of the sort that are commonly encountered in the literature there is a natural scaling that can be easily computed that makes the magnitude of the Evans function a lower bound for the spectral distance function. In this scaling a zero of the Evans function still represents an eigenvalue, but a non-zero value of the Evans function $E(\lambda)$ guarantees that a disk of radius $|E(\lambda)|$ in the complex plane is free of eigenvalues. Further this extends in a straightforward way to dependence of the Evans function on parameters. We illustrate our main results with some numerical experiments.  Throughout this paper we assume that the operator ${\mathcal{L}}$ is a differential operator on a interval, satisfying either separated or periodic boundary conditions, and that  ${\mathcal{L}}$ has compact resolvent. As is common, $\rho(\mathcal{L})$ will denote the resolvent set and  $\sigma(\mathcal{L})$ the spectrum. 

Consider the simplest setting, where the boundary conditions are separated. Typically, the eigenvalue equation $\mathcal{L} u = \lambda u$ in some specified domain is equivalent to an ODE on $\mathbb{R}^n$ of the form
\[
u_x = A(x,\lambda) u \qquad\qquad B^{(L)} u(0) = 0 \qquad B^{(R)} u(L) = 0.
\]
Here, $B^{(L)},B^{(R)}$ are $k\times n$ and $(n-k)\times n$ matrices imposing the left and right boundary conditions, respectively, assumed to be of full rank. It will be assumed that the matrix $A(x,\lambda)$ is traceless, and thus the Wronskian is independent of $x$. This is usually true in practice and can always be achieved by an appropriate change of variables. The Evans function may be defined as
\[
E(\lambda) = \det(u^{(L)}(0),u^{(R)}(0)).
\]
Here $u^{(L)}(x)$ is a $(n-k)\times n$ solution to $u_x^{(L)} =A(x,\lambda) u^{(L)}$ where $u^{(L)}(0)$ is a basis for $\nls(B^{(L)})$ and similarly $u^{(R)}(x)$ is a $k\times n$ solution to $u_x^{(R)} =A(x,\lambda) u^{(R)}$ with $u^{(R)}(L)$ a basis for $\nls(B^{(R)}).$ Under very general conditions $E(\lambda)$ vanishes if and only if $\lambda$ is an eigenvalue of the boundary value problem, and the degree of the zero of $E(\lambda)$ is equal to the multiplicity of the eigenvalue. It is well understood that the Evans function encodes topological information regarding the intersection of the spaces of solutions to the differential equation satisfying the left and right boundary conditions. One manifestation of this can be seen in the above construction: $u^{(R)}(L)$ and $u^{(L)}(0)$ can be chosen to be any bases for $\nls(B^{(R)})$ and $\nls(B^{(L)})$, respectively. Different choices lead to different functions $E(\lambda)$ which nevertheless share the same zero set. Several authors have exploited the invariance of the Evans function under a change of basis to develop numerical methods with desirable properties, most notably removing numerical stiffness issues by working over the Steifel manifold\cite{AB2010,HZ2006}

In sections \ref{sec:secondorder} and \ref{sec:thmsandcalc} we make the main observation that for many of the sets of boundary conditions that arise in practice one can choose the normalization of the Evans function in a natural way that encodes more quantitative information  about the intersection of the left and right subspaces. In particular one can choose a normalization such that, not only is $\lambda$ an eigenvalue if and only if $E(\lambda)=0,$ but for $\lambda\in\rho({\mathcal L})$ the magnitude of $E(\lambda)$ is a lower bound for the spectral distance function: if $\lambda \in \rho(\mathcal{L})$ one is guaranteed that the ball 
\[
D_{\lambda}=\{\lambda^* ~:~ \vert \lambda - \lambda^* \vert < |E(\lambda)|~\}
\] necessarily lies in the resolvent set. In section \ref{sec:numexam} we demonstrate that this normalized Evans function can be easily computed numerically, and that it can be used as the basis for more efficient root finding algorithms, since a single computation allows one to eliminate an open subset of $\lambda$ values from consideration. 

\section{A second order problem} \label{sec:secondorder}
Throughout this paper we will adopt the following conventions: $\mathcal{L}$ will denote a linear differential operator with smooth bounded coefficients together with an appropriate set of boundary conditions. 
\begin{definition}
    Given an operator $\mathcal{L}$ and a $\lambda\in \mathbb{C}$ we can define the resolvent as the operator $(\mathcal{L} - \lambda)^{-1}$ whenever this operator exists as a bounded operator on $L_2(I).$ We say that $\lambda$ is in the resolvent set of $\mathcal{L}$ wherever $(\mathcal{L} - \lambda)^{-1}$ exists in this sense, and the spectrum is the complement of the resolvent set.
\end{definition}

Throughout this paper the operators $\mathcal{L}$ considered will be ordinary differential operators on a compact interval. As such they have, under very mild hypotheses, compact resolvent, 
and this will generally be assumed throughout the rest of the paper.
The next lemma is well-known, and will be used throughout the paper.

\begin{lemma}
\label{lem:SpectralDistance}
    If a point $\lambda^*$ lies in the resolvent set $\lambda \in \rho(\mathcal{L})$ and $d(\lambda^*)$ denotes the distance from $\lambda^*$ to the spectrum $\sigma(\mathcal{L})$ then $$d(\lambda^*) \geq \Vert(\mathcal{L}-\lambda^*)^{-1}\Vert^{-1}$$
\end{lemma}
\begin{proof}
This follows immediately from the convergence of the geometric series  
\[
{(\mathcal{L}-\lambda)}^{-1} = \sum_{k=0}^\infty (\lambda-\lambda^*)^{k}{(\mathcal{L}-\lambda^*)}^{-(k+1)}
\]
in the open disk $|\lambda-\lambda^*| < \Vert(\mathcal{L}-\lambda^*)^{-1}\Vert^{-1}.$ 
\end{proof}

\subsection{Main Idea}
Consider the second order eigenvalue problem 
\[
\left(\mathcal{L} - \lambda\right)y = 0
\]
where $\mathcal{L}$ is a second order differential operator on $\mathbb{R}$ with separated homogeneous boundary conditions at $x=0$ and $x=L.$ 

Let $u^{(L)}(x;\lambda)$ denote a non-zero solution to the second order differential equation
\[(\mathcal{L}-\lambda)u^{(L)}(x;\lambda) \]
satisfying the boundary condition at $x=0$, and similarly let $u^{(R)}$ be a solution to the differential equation satisfying the boundary condition at $x=L$. The Evans function in this case is given by the Wronskian evaluated at $x=0,$ $E(\lambda)=u^{(L)}(0;\lambda)u^{(R)}_x(0;\lambda)-u^{(R)}(0;\lambda)u^{(L)}_x(0;\lambda).$\footnote{Properly speaking there are many choices of the Evans function, depending on the scaling chosen and the point at which the Wronskian determinant is evaluated.} For ease of exposition we will assume that the ODE is area preserving and the Wronskian of any two solutions is constant. This is not necessary but simplifies the calculation somewhat, and can always be brought about by a simple change of variables. 

Note that a straightforward calculation using variation of parameters or the Green's function shows that 
\[
(\mathcal{L}-\lambda)^{-1} f = \frac{1}{E(\lambda)}\left(u^{(R)}(x;\lambda)\int_{0}^x u^{(L)}(y;\lambda) f(y) dy + u^{(L)}(x;\lambda) \int_x^L u^{(R)}(y;\lambda) f(y) dy\right)\]
\[=\frac{1}{E(\lambda)} \int_{0}^L G(x,y) f(y) dy 
\]
The Cauchy-Schwartz inequality gives the obvious bound
\[
\Vert(\mathcal{L}-\lambda)^{-1}\Vert_{L_2} \leq \frac{1}{|E(\lambda)|} \left(\iint_{[0,L]^2} G^2(x,y;\lambda) dx dy\right)^{\frac12}
\]
From this it follows that the spectral distance function $d(\lambda)$ satisfies
\[
d(\lambda)\geq \Vert(\mathcal{L}-\lambda)^{-1}\Vert_{L_2}^{-1} \geq \frac{|E(\lambda)|} {\left(\iint_{[0,L]^2} G^2(x,y;\lambda) dx dy\right)^{\frac12}} := \frac{|E(\lambda)|}{W(\lambda)}
\]
The standard estimate in lemma \ref{lem:SpectralDistance} shows that the reciprocal of the norm of the resolvent is a lower bound for the spectral distance function, so if one scales the Evans function by the Hilbert-Schmidt norm of the integral kernel for the Green's functions that provides a lower bound on the spectral distance function. Note that this scaling is natural, in the sense that the resulting Evans function is independent of the scaling of the left and right solutions $y_1$ and $y_2$. In other words under the rescaling of $(u^{(L)},u^{(R)}) \mapsto (\alpha u^{(L)},\beta u^{(R)})$ for $\alpha,\beta \neq 0$, the rescaled Evans function $\frac{E(\lambda)}{W(\lambda)}$ is invariant. Later we will see that in the higher order in which $d$ boundary conditions are imposed at $x=0$ and $n-d$ boundary conditions are imposed at $x=L$ the normalized Evans function is constant on connected components of $\mathsf{GL}(d,{\mathbb R})\times \mathsf{GL}(n-d,\mathbb{R}).$

Further this quantity is a lower bound for the spectral distance function: given a point in the resolvent set $\lambda^*\in\rho(\mathcal{L})$ the ball
\[
D_\lambda = \{ \lambda^* : |\lambda-\lambda^*| <\frac{|E(\lambda^*)|}{W(\lambda^*)}\}
\]
must also lie in the resolvent set. 
In the following section we provide some similar calculations for the case of higher order separated boundary conditions as well as periodic boundary conditions. While these are certainly not the best possible bounds, for reasons to be discussed, they are quite general. 

%\subsection{Scaling of the Evans function}
%A useful property of the Evans function is that it can be scaled while preserving it's analyticity and zeros. Scaling the Evans function by any non-zero constant, or more generally any non-zero analytic function leaves these properties unchanged. In this regard there is no single Evans function associated with any differential operator, different authors will often refer to different functions as \underline{the} Evans function. \\

%\subsection{A quasi-Newton iteration}
This construction suggests a quasi-Newton iteration for finding the eigenvalues of a self-adjoint operator. Since the calculation of the Evans function at a point $\lambda^*$ guarantees that the interval 
\[
(\lambda^* -\frac{|E(\lambda^*)|}{W(\lambda^*)}, \lambda^* +\frac{|E(\lambda^*)|}{W(\lambda^*)} )
\]lies in the resolvent set it immediately suggests an 'adaptive' stepping scheme for finding the eigenvalues. 
Given a real starting point $\lambda_0$ consider the following two iterations with common starting value $\lambda_0$:
\begin{align*}
&
\lambda_{i+1}^{+} = \lambda_i^{+} + \frac{E(\lambda_{i}^{+})}{W(\lambda_i^+)} \\
&
\lambda_{i+1}^{-} = \lambda_i^{-} - \frac{E(\lambda_{i}^{-})}{W(\lambda_i^-)}.
\end{align*}
Note that $E(\lambda_i^\pm)$ necessarily has the same sign for all $i$, which we will henceforth assume to be positive. In this case the sequences $\{\lambda_i^+\}_{i=0}^\infty$ and $\{\lambda_i^-\}_{i=0}^\infty$ are increasing and decreasing respectively. One expects (and we will show) that these iterations will converge to the smallest eigenvalue greater than $\lambda_0$ and the largest eigenvalue smaller than $\lambda_0$ respectively. We will also show that for the case of a second order equation with separated boundary conditions these iterations are quasi-Newton, in the sense that in a neighborhood of an eigenvalue $\lambda$
\[
\frac{|E(\lambda_{i})|}{W(\lambda_i)} = |\lambda_i-\lambda| + O(|\lambda_i-\lambda|^2)
\]
which is sufficient to guarantee quadratic convergence of the iteration. 

\begin{theorem}
    Suppose $\mathcal{L}$ is a real self adjoint ordinary  differential operator and let $\lambda_0 \in \rho(\mathcal{L})$. Assume that $E(\lambda_0)>0$ and define the sequences 
    \begin{equation*}
        \{\lambda^\pm_i\} = \begin{cases}
        \lambda_0 & i = 0\\
        \lambda_{i-1} \pm E(\lambda_{i-1})/W(\lambda_{i-1}) & i > 0,
    \end{cases}\end{equation*}
Now let $\lambda^+ = \min\{x\in \sigma(\mathcal{L})\,|\, x > \lambda_0\}$ and $\lambda^- = \max\{x\in \sigma(\mathcal{L})\,|\, x < \lambda_0\}$. If $\lambda^\pm$ exist then they are the limits of $\{\lambda_i^{\pm}\}$ respectively.
\label{thm:QNewton}
\end{theorem}

\begin{pf}
    We prove the case for $\{\lambda^+_i\}$, under the assumption that $E(\lambda_1)>0$; the other cases follow mutatis mutandis. Every element of the sequence satisfies $\lambda_i = \lambda_{i-1} + E(\lambda_{i-1})/W(\lambda_{i-1}) \leq \lambda_{i-1} + d(\lambda_{i-1}) \leq \lambda^+$, the first inequality follows from Theorem 2.1 and the second from the definition $\lambda^+$. Since the sequence increasing it must have a limit $\tilde{\lambda}\leq\lambda^+$. Now 
    $$
        \lim_{i\to \infty} \frac{E(\lambda_i)}{W(\lambda_i)} =  \lim_{i\to \infty} (\lambda_{i+1}-\lambda_{i}) = 0  
    $$
    however $\frac{E}{W}$ is continuous, so 
   $$\lim_{i\to \infty} \frac{E(\lambda_i)}{W(\lambda_i)} = \frac{E(\lambda)}{W(\lambda)} = 0.$$
   The second equality is only the case if $E(\lambda) = 0$ so we conclude that $\lambda = \lambda^+$.
\end{pf}

%\noindent \textbf{Remark} The above result has an analog for $\mathcal{L}$ not self adjoint. One can construct sequence $\{\lambda^{\pm}_i\}$ along any line that passes through $\lambda$. \textcolor{magenta}{Not sure exactly what is being claimed here but I'm not sure it is right. If the spectrum is discrete a typical line with not have any eigenvalues that lie on it.  }

The next result is specific to the case of a second order equation with separated boundary conditions, and shows that in a neighborhood of an eigenvalue $\lambda^*$ the normalized Evans function is, to leading order, equal to the spectral distance function. An obvious consequence of this is that the iteration above is quasi-Newton, in the sense convergence of the iteration to the eigenvalue is quadratic. We state the following result here, although we defer the proof to a later section.   

\begin{theorem} \label{thm:derivative}
    For a second order operator with separated boundary conditions, then for $\lambda \in \sigma(\mathcal{L})$
    \begin{equation}
        \frac{d}{d\lambda}\frac{E(\lambda)}{W(\lambda)} = \pm 1.
    \end{equation}
\end{theorem}
\begin{proof}
    see following section
\end{proof}

Under very mild assumptions the Evans function is an analytic function of the eigenvalue parameter. Assuming that the weight $W(\lambda)$ is also chosen to be analytic, as will be the case in the construction to follow,  it is clear that in a neighborhood of an eigenvalue, $E(\lambda)=0$,  we have that 
$$\frac{|E(\lambda^*)|}{W(\lambda^*)}=|\lambda^*-\lambda| + O(|\lambda^*-\lambda|^2).$$ This is exactly the relation that must hold in order to have quadratic convergence, as in the Newton iteration. If such a relation holds then for $|\lambda_i-\lambda|\ll 1$ we have that 
\[
|\lambda_{i+1}-\lambda| = O(|\lambda_{i}-\lambda|^2)
\]
In other words quadratic convergence, as in the classical Newton iteration. 

\section{Theorems and Basic Calculations} \label{sec:thmsandcalc}

This section will establish a class of `natural' scalings for the Evans function, which we shall refer to as the weighted Evans function. We will show the weighted Evans function satisfies the property that it's magnitude is always less than the magnitude of the spectral distance function. Finally, we will demonstrate the calculation of the weighted Evans function in some examples.

\subsection{Main Theorem}

We can relate the Evans function to the norm of the resolvent operator using a scaling factor and we recover the lower bound on spectral distance through application of lemma \ref{lem:SpectralDistance}.

\begin{prop}
    For a class of differential operators $\mathcal{L}$ including ones equipped with periodic or separated boundary conditions, there is a non-zero analytic function $W(\lambda)$ such that $\frac{W(\lambda)}{|E(\lambda)|} \geq \left\lVert (\mathcal{L} - \lambda)^{-1}\right\rVert$. \footnote{Here the inequality denotes that if both sides exist then the relation is satisfied and if one side fails to exist then the other side fails to exist as well.} We will call this function the \underline{weighting function}. The function $\frac{E(\lambda)}{W(\lambda)}$ is the weighted Evans function.
\end{prop}

\subsection{Weighting Function for Separated Boundary Conditions} \label{subsec:separatedderivation}

We consider a differential operator $\mathcal{L}$ with homogenous separated boundary conditions. 
We write the associated eigenvalue problem as a system of $n$ differential equations: $u' = A(x,\lambda) u + \vec{f}$ where $A(x,\lambda) \in M_{n\times n}$ and $\vec{f}$ has the function $f$ in the last coordinate and zeros otherwise. We assume that $k$ boundary conditions are imposed at $x=0$ and $n-k$ boundary conditions are imposed at $x=L$. In typical applications $n$ is even and $k=\frac{n}{2}$ although we do not necessarily assume this.  The boundary conditions can be written as $B_{L}u(0) = 0$ and $B_{R}u(L) = 0$ where $B_{L}\in M_{k \times n}$ and $B_R\in M_{n-k\times n}$ for some $0 < k < n$ such that $B_{L}B_{L}^T = I_k$ and $B_{R}B_{R}^T = I_{n-k}$. 
By variation of parameters, the solution to the system is,
$$u(x) = U(x) \int_{0}^{x} U^{-1}(y)f(y)dy + U(x)u_{0}$$
where $U$ is the fundamental matrix solution.
Applying the left boundary condition, 
$$B_{L}u(0) = B_{L}u_{0} = 0.$$
So we can write,
$$u_{0} = N_{L}\tilde{u}_{0}$$
where $N_{L}$ is a basis for the nullspace of $B_{L}$.
Applying the right boundary condition,
$$B_{R}u(L) = B_{R}M\int_{0}^{L} U^{-1}(y)f(y)dy + B_{R}Mu_{0} = 0$$
$$B_{R}Mu_{0} = - B_{R}M\int_{0}^{L} U^{-1}(y)f(y)dy,$$
where $M = U(L)$ is the monodromy matrix.
Substituting the first equation,
$$B_{R}MN_{L}\tilde{u}_{0} = - B_{R}M\int_{0}^{L} U^{-1}(y)f(y)dy.$$
To guarantee a unique solution for $u_{0}$ we need $B_{R}MN_{L}$ invertible, and its determinant is thus an Evans function. So,
$$\tilde{u}_{0} = -(B_{R}MN_{L})^{-1}B_{R}M\int_{0}^{L} U^{-1}(y)f(y)dy,$$
$$u_{0} = - N_{L}(B_{R}MN_{L})^{-1}B_{R}M\int_{0}^{L} U^{-1}(y)f(y)dy.$$
We can write
or 
$$u_{0} = - Q\int_{0}^{L} U^{-1}(y)f(y)dy$$
by letting $Q = N_{L}(B_{R}MN_{L})^{-1}B_{R}M$.  
Altogether we have,
$$u(x) = U(x)(I-Q) \int_{0}^{x} U^{-1}(y)f(y)dy - U(x)Q\int_{x}^{L} U^{-1}(y)f(y)dy.$$
Defining $P$ by\footnote{where $E(\lambda) = 0$ we define $P$ either by the limit or taking adjugates of $B_RMN_L$}
$$Q = \frac{1}{E(\lambda)}P$$
and substituting back in, we obtain the solution,
$$u(x) = \frac{U(x)}{E(\lambda)}\left[(E(\lambda)-P)\int_{0}^{x} U^{-1}(y)f(y)dy - P\int_{x}^{L} U^{-1}(y)f(y)dy \right].$$
Now we let 
$$\mathcal{G}(x,y,\lambda) = \begin{cases} 
    U(x)(E(\lambda)-P)U^{-1}(y) & y < x\\
    -U(x)PU^{-1}(y) & y > x.\\
    \end{cases}$$
And we rewrite the solution as,
$$u(x) = \frac{1}{E(\lambda)}\int_{0}^{L} \mathcal{G}(x,y,\lambda)f(y)dy.$$
Since we want the resolvent map we consider the $1$st coordinate of $u$,
$$u_1(x) = \frac{1}{E(\lambda)}\int_{0}^{L} \sum_i[\mathcal{G}(x,y,\lambda)]_{1i}*[f(y)]_idy.
$$
This is 
$$u_1(x) = \frac{1}{E(\lambda)}\int_{0}^{L} [\mathcal{G}(x,y,\lambda)]_{1n}f(y)dy.$$
Letting 
$$G = \mathcal{G}_{1n}.$$
we have 
$$u_1(x) = \frac{1}{E(\lambda)}\int_{0}^{L} G(x,y) f(y) dy.$$
By the Cauchy-Schwarz inequality we have a pointwise estimate
$$|u_1(x)|\leq \frac{1}{|E(\lambda)|} \left(\int_{0}^{L} G^2(x,y)dy\right)^{\frac12} \Vert f \Vert_2  $$
which immediately gives 
$$\Vert u_1\Vert_2  \leq \frac{1}{|E(\lambda)|}\left(\iint_0^L G^2(x,y) dx dy\right)^{\frac12} \Vert f \Vert_2$$
\begin{equation}\label{eq:resolventestimate}
    \left\lVert (\mathcal{L} - \lambda)^{-1} \right\rVert  \leq \frac{1}{|E(\lambda)|}\left(\int_{0}^{L}\int_{0}^{L} G^2dydx\right)^{\frac{1}{2}}. 
\end{equation}
So a choice weighting function is,
\begin{equation}\label{eq:WeightedEF}W(\lambda) = \left(\int_{0}^{L}\int_{0}^{L} G^2dydx\right)^{\frac{1}{2}}.\end{equation}

Having derived this representation of the Green's function we can now prove the result stated in \ref{thm:derivative}:

\begin{pf}[Theorem \ref{thm:derivative}]
Let $\lambda$ be an eigenvalue. We first calculate the weighting function. Since $E(\lambda) = 0$ we have $\mathcal{G}(x,y,\lambda) = - U(x)PU^{-1}(y)$. 
Now, 
$$G(x,y,\lambda) = [\mathcal{G}(x,y,\lambda)]_{12} = -\sum_i [U(x)N_L]_{1i}[P'B_RMU^{-1}(y)]_{i2}.$$
Here $P'$ is defined such that $\frac{1}{E(\lambda)}P' = (B_RMN_L)^{-1}$ and using the limit when $E(\lambda) = 0$. 
Looking at the dimension of these matrices we have $B_R \in M_{1\times 2},\, M \in M_{2\times 2}$ and $N_L \in M_{2\times 1}$, so $B_RMN_L \in M_{1\times 1}$ and 
$$(B_RMN_L)^{-1} = \frac{1}{\det(B_RMN_L)} = \frac{1}{E(\lambda)}.$$ We conclude that $P' = 1$ so, 
$$G(x,y,\lambda) = -\sum_i [U(x)N_L]_{1i}[B_RMU^{-1}(y)]_{i2}.$$
Further $U(x)N_L \in M_{2\times 1}$ and $B_RMU^{-1}(y) \in M_{1\times 2}$ so we can simply write,
$$G(x,y,\lambda) = -[U(x)N_L]_{1}[B_RMU^{-1}(y)]_{2}.$$
Now we note that $U(x)N_L$ is the unique solution satisfying the left boundary conditions as $B_LU(0)N_L = B_LN_L = 0$. So let $u_L \coloneqq [U(x)N_L]_{1}$. In turn $B_RMU^{-1}(y)$ is of the form $(-u'_R,\, u_R)$ for some function $u_R$ satisfying the differential equation (since we are multiplying $U^{-1}(y)$ by a $1\times 2$ matrix on the left) and since $B_RMU^{-1}(L)N_R = B_RIN_R = B_RN_R = 0$ we find that $u_R$ is the unique solution satisfying the right boundary conditions (since $(-u'_R,\, u_R)N_R = B_R(u_R,u'_R)^T$). So we have, 
$$G(x,y,\lambda) = -u_L(x)u_R(y).$$
However, since $\lambda$ is an eigenvalue, there is a unique solution satisfying both sets of boundary conditions and so we conclude that $u_L$ and $u_R$ are multiples of each other. We write, 
$$G(x,y,\lambda) = -ku_L(x)u_L(y).$$
Thus, 
$$W(\lambda) = \left(\int_0^L\int_0^L k^2u_L(x)^2u_L(y)^2dydx\right)^{1/2}$$
$$ = \left(\int_0^Lku_L(x)^2dx\int_0^L ku_L(y)^2dy\right)^{1/2}$$
$$= \left|\int_0^L u_L(y)u_R(y) dy\right|.$$

Now we will calculate $\frac{d}{d\lambda}E(\lambda).$ We note that 
$(u_1)_\lambda$ satisfies the differential equation, 
$$Au - \lambda u = u_1$$ with boundary conditions $(u_1)_\lambda(0) = (u_1)_\lambda'(0) = 0.$ By the variation of parameters formula we obtain, 
$$(u_1)_{\lambda} = \int_0^x [U(x)U^{-1}(y)]_{12}u_1(y)dy$$
$$(u_1)'_{\lambda} = \int_0^x [U(x)U^{-1}(y)]_{22}u_1(y)dy.$$
So letting $U_{\lambda}(x) \coloneqq \frac{d}{d\lambda} U(x)$ we have, 
$$[U_\lambda]_{ij} = \int_0^x[U(x)U^{-1}(y)]_{i2}[U(y)]_{1j}dy$$
$$[M_\lambda]_{ij} = \int_0^L[MU^{-1}(y)]_{i2}[U(y)]_{1j}dy.$$
Now we can think of $M_{\lambda}$ as a the product of vectors, 
$$M_\lambda = \int_0^L[MU^{-1}(y)]_{*2}[U(y)]_{1*}dy.$$
Then, 
$$B_RM_{\lambda}N_L = \int_0^LB_R[MU^{-1}(y)]_{*2}[U(y)]_{1*}N_Ldy = \int_0^L[B_RMU^{-1}(y)]_{*2}[N_LU(y)]_{1*}dy.$$
Now these are each scalars so we can drop the second subscript,
$$B_RM_{\lambda}N_L = \int_0^L[B_RMU^{-1}(y)]_{2}[N_LU(y)]_{1}dy = \int_0^L u_R(y)u_L(y)dy.$$
Since $B_RM_{\lambda}N_L$ is a scalar, 
$$\frac{d}{d\lambda}E(\lambda) = \det(B_RM_{\lambda}N_L) = \int_0^L u_R(y)u_L(y)dy.$$

Finally we compute the derivative of the weighted Evans function, 
$$\frac{d}{d\lambda} \frac{E(\lambda)}{W(\lambda)} = \frac{W(\lambda)\frac{d}{d\lambda}E(\lambda) - E(\lambda)\frac{d}{d\lambda}W(\lambda)}{W(\lambda)^2} = \frac{\frac{d}{d\lambda}E(\lambda)}{W(\lambda)}.$$
This is, 
$$=\frac{\int u_Lu_Rdy}{|\int u_Lu_R dy|} = \pm 1.$$
\end{pf}

\subsection{Weighting Function for Periodic Boundary Conditions}

Next we derive an appropriate weighting function for the Evans function with periodic boundary conditions. Let $U(x)$ be the fundamental matrix solution to our differential equation. 
Then the variation of parameters solution to $u' = \mathcal{L}u + f$ is 
$$w(x) = U(x)\int_{0}^{x}U^{-1}(y)f(y)dy + U(x)w_{0}.$$
Now we apply the periodic boundary conditions, 
$$w(0) = w_{0}$$
$$W(L) = M\int_{0}^{L}U^{-1}(y)f(y)dy + Mw_{0}.$$
Setting both expressions equal to each other,
$$M\int_{0}^{L}U^{-1}(y)f(y)dy + Mw_{0} = w_{0}$$
$$(M-I)w_{0} = -M\int_{0}^{L}U^{-1}(y)f(y)dy.$$
Plugging back into the variation of parameters solution, it is not hard to see that if we want a unique solution for $w(x)$ to exist we need $(M-I)$ invertible. In fact, the Evans function will be $\det(M-I)$. 
We let $(M-I)^{-1} \eqqcolon \frac{1}{E(\lambda)}Q$ and $P = QM$ then 
$$w(x) = \frac{U(x)}{E(\lambda)}\left[(E(\lambda)-QM)\int_{0}^{x}U^{-1}(y)f(y)dy - QM\int_{x}^{L}U^{-1}(y)f(y)dy\right].$$
Finally, as before we let 
$$\mathcal{G}(x,y,\lambda) = \begin{cases} 
    U(x)(E(\lambda)-P)U^{-1}(y) & y < x\\
    -U(x)PU^{-1}(y) & y > x\\
    \end{cases}$$
so that 
$$w(x) = \frac{1}{E(\lambda)}\int_{0}^{L} \mathcal{G}(x,y,\lambda)f(y)dy.$$
From here the same procedure with Cauchy-Schwarz can easily be applied.

\subsection{A remark on systems and an identity.}

In each of the previous examples we wrote the solution to a $n^{th}$ order equation by rewriting it as a system. In this case the forcing term takes the special form where only the last component is non-zero. For completeness sake we present a result for general systems, which uses an interesting matrix inequality. 

\iffalse
\begin{definition} Let the Hilbert-Schmidt norm of a matrix be ${\left\lVert A \right\rVert}_{H,S} = \left(\sum_{i} \langle Ae_{i},Ae_{i}\rangle\right)^{\frac{1}{2}}$. 
\end{definition}

\noindent \textbf{Fact.}
    For any matrix $A$ we have ${\left\lVert A \right\rVert}_{2} \leq {\left\lVert A \right\rVert}_{HS}$. This is not hard to show using the Cauchy-Schwarz inequality.
\fi
\begin{lemma}
\label{thm:HS}
    Let $A$ be a nonsingular $d\times d$ matrix, and $\Vert A \Vert_{HS}$ the usual Hilbert-Schmidt norm, $\Vert A \Vert_{HS}^2=\sum_{i,j} |A_{ij}|^2.$ Then we have the following inequality relating the Hilbert-Schmidt norm of the matrix to that of the inverse matrix.
    $$
        \left\lVert A^{-1} \right\rVert_{HS} \leq \frac{(\left\lVert A \right\rVert_{HS})^{d-1}}{\vert\det(A)\vert d^{\frac{d-2}2}}.   
        $$
\end{lemma}

\begin{pf}
See appendix.
\end{pf}

\begin{remark}
    We note a couple of things here. First, equality holds in the above for the identity matrix. Secondly, this inequality is somewhat inefficient. The variation of parameters formula involves integrating a matrix element of the matrix $G(x,y,\lambda)$; this element can be bound by the norm the above theorem used to approximate the norm of the inverse, explicitly: 
    $$\left\lVert (\mathcal{L}-\lambda)^{-1} \right\rVert \leq \frac{1}{d^{d-2}|E(\lambda)|} \int\int{\left\lVert U(x)P\right\rVert}_{H,S}^{2} {\left\lVert U(y)\right\rVert}_{H,S}^{2d-2} dy$$
    where $G(x,y,\lambda) = U^{-1}(y)PU(x)$. As one might expect, the result is unfortunately not particularly tight.
\end{remark}

\subsection{Extension of parameters} \label{subsec:extensionofparams}

Often in the study of stability problems one is interested in the dependence on a parameter. It may be, for example, that one is studying the stability of a family of traveling waves that depends on one or more parameters. Another example would be to understand the dependence of the spectrum of a periodic operator on the Floquet exponent. This is an example that we will consider in the next section, but we remark here that one can, with suitable modification, also use the normalized Evans function to guarantee that a region of parameter space is free from eigenvalues. The basic idea is that if the norm of the resolvent composed with the perturbation is sufficient small then one can construct the inverse of the perturbed operator in terms of a geometric series. Specifically we have the well-known

\begin{lemma}\label{extension}
    Let $\mathcal{L},P$ be operators such that $\mathcal{L}^{-1}$ exists.
If $|\mu| * \left\lVert P(\mathcal{L}^{-1})  \right\rVert  < 1$ or $|\mu| * \left\lVert (\mathcal{L}^{-1})P  \right\rVert  < 1$ then $(\mathcal{L} - \mu P)^{-1}$ exists. 
\end{lemma} 

\begin{pf} 
Suppose that $|\mu| * \left\lVert P(\mathcal{L}^{-1})  \right\rVert  < 1$.
The following power series thus converges, 
$$
    \sum_{n=0} [\mu P(\mathcal{L}^{-1})]^n = (1 - \mu P(\mathcal{L}^{-1}))^{-1}.
$$
Now, 
$$(\mathcal{L})^{-1}(1 - \mu P(\mathcal{L}^{-1}))^{-1} = [(1 - \mu P(\mathcal{L}^{-1}))\mathcal{L}]^{-1}$$
$$ = [\mathcal{L} - P]^{-1}.$$
So the operator exists. In turn the case with $|\mu| * \left\lVert (\mathcal{L}^{-1})P  \right\rVert  < 1$ proceeds by multiplying 
$$
    (1 - \mu (\mathcal{L}^{-1})P)^{-1}(\mathcal{L})^{-1}.$$
\end{pf}

\begin{example}
    Consider the bounded eigenvalue problem, 
$$-\partial_{xx}u + Q(x)u = \lambda u$$
where $Q(x) = Q(x+L)$. We find that $-\partial_{xx} + (Q(x)-\lambda)$ is periodic in $L$, so by Floquet's theorem we have a solution, 
$$u(x) = e^{\eta x}\phi(x)$$
with $\phi(x) = \phi(x+L)$ and $\eta$ complex. Since we want bounded solutions, $\eta$ cannot have a real part. So we have,
$$u(x) = e^{i\mu x}\phi(x)$$
with $\mu \in [-\pi/L,\pi /L)$. Now we plug this back into the differential equation,
$$-\partial_{xx}e^{i\mu x}\phi(x) + Q(x)e^{i\mu x}\phi(x) = \lambda e^{i\mu x}\phi(x)$$
$$-\left[-\mu^{2}e^{i\mu x}\phi(x) + 2i\mu e^{i\mu x}\phi(x)' + e^{i\mu x}\phi(x)''\right] + Q(x)e^{i\mu x}\phi(x) = \lambda e^{i\mu x}\phi(x)$$
$$-\left[-\mu^{2}\phi(x) + 2i\mu \phi(x)' + \phi(x)''\right] + Q(x)\phi(x) = \lambda\phi(x)$$
$$-(\partial_{x} + i\mu)^{2}\phi(x) + Q(x)\phi(x) = \lambda\phi(x).$$
Letting
$\mathcal{L}[\mu] = -(\partial_{x} + i\mu)^{2} + Q(x),$ we can write 
$$\mathcal{L}[\mu + \Delta \mu] - (\lambda + \Delta \lambda) = \mathcal{L}[\mu] - (\lambda) - 2i(\Delta \mu)\partial_{x} + 2\mu*\Delta\mu + (\Delta \mu)^{2} - \Delta\lambda.$$ 
So supposing that $(\mathcal{L}[\mu] - \lambda)^{-1}$ exists, applying lemma 2.3, if 
$$2|\Delta \mu|\left\lVert \partial_{x}(\mathcal{L}[\mu]- \lambda )^{-1} \right\rVert + \left|(\Delta \mu)^{2} + 2\mu*\Delta\mu - \Delta \lambda\right|\left\lVert (\mathcal{L}[\mu]- \lambda )^{-1} \right\rVert < 1$$
then $(\mathcal{L}[\mu + \Delta \mu] - (\lambda + \Delta \lambda))^{-1}$ exists.

Of particular interest is the calculation of $\left\lVert \partial_{x}(\mathcal{L}[\mu]- \lambda )^{-1} \right\rVert$. By a modification to \ref{subsec:separatedderivation} if 
$$\left\lVert (\mathcal{L}[\mu]- \lambda )^{-1} \right\rVert  \leq \frac{1}{|E(\lambda)|}\left(\int_{0}^{L}\int_{0}^{L} Gdydx\right)^{\frac{1}{2}}$$
according to (\ref{eq:resolventestimate})
then
$$\left\lVert \partial_{x}(\mathcal{L}[\mu]- \lambda )^{-1} \right\rVert  \leq \frac{1}{|E(\lambda)|}\left(\int_{0}^{L}\int_{0}^{L} (\partial_xG)^2dydx\right)^{\frac{1}{2}}.$$
\end{example}

\section{Some numerical examples} \label{sec:numexam}

In this section we present some numerical experiments to illustrate the ideas. 

\subsection{Mathieu equation with Dirichlet boundary conditions}

We begin by computing the weighted Evans for the Dirichlet Mathieu operator $\mathcal{L} = -\partial_{xx} + \cos(x)$ with Dirichlet boundary conditions on $(0,2\pi)$. The operator is self-adjoint so that the eigenvalues lie on the real axis. We will focus on the interval $[2,5]$, as depicted in figure \ref{fig:Fig1}. The blue curve depicts the Evans function computed by smapling. One can see that there are two eigenvalues in this interval, $\lambda\approx 2.3$ and $\lambda\approx 4.0.$ Computing the weighted Evans function at the point $\lambda=3.0$ we find that the magnitude of the weighted Evans function is roughly $0.51$, guaranteeing that the interval $(\sim2.49,\sim3.51)$ is free of eigenvalues.  

\begin{figure}[H]
\centering

  \includegraphics[width=.7\linewidth]{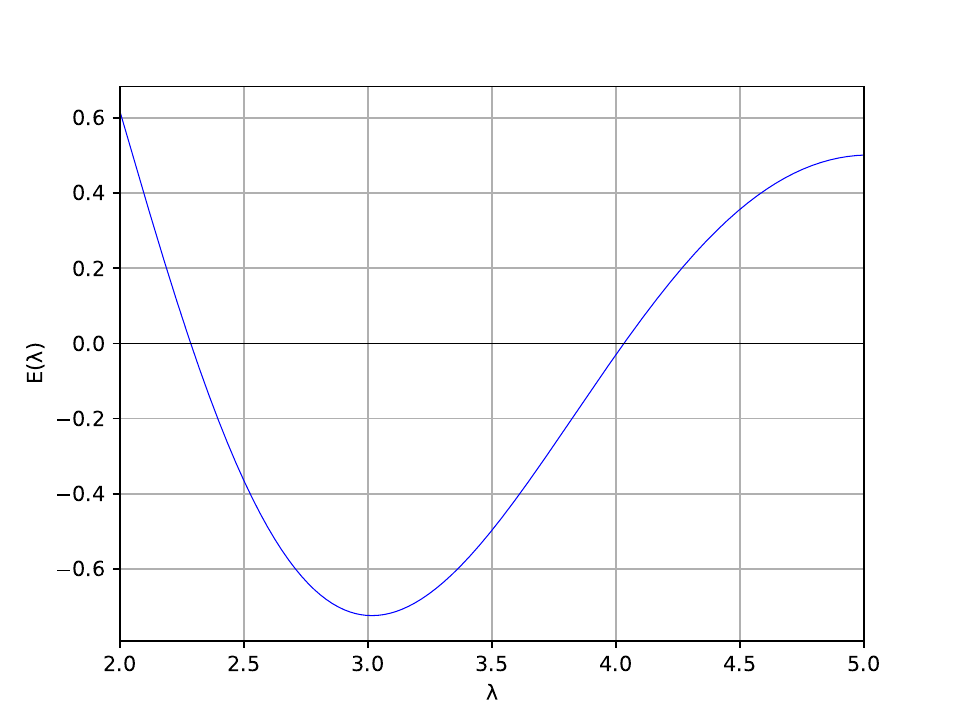}
  \caption{\footnotesize Evans function on $[2,5]$}
  \label{fig:Fig1}
\end{figure}

\noindent Now we compute the weighted Evans function according to (\ref{eq:WeightedEF}) at $\lambda = 3$.

\begin{figure}[H]
    \centering
    \includegraphics[width=.7\linewidth]{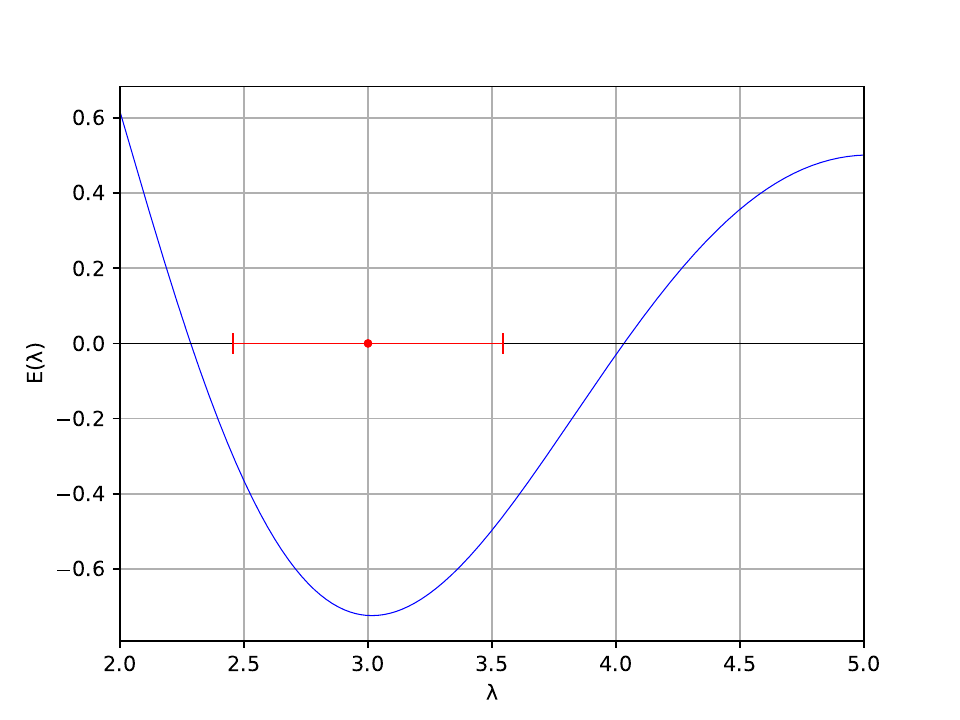}
    \caption{\footnotesize Evans function on $[2,5]$ with weighted Evans function calculated at $\lambda = 3$.}
    \label{fig:Fig2}
\end{figure}

One can easily see that by computing the weighted Evans function, one finds an eigenvalue free region around the specified value of $\lambda$. We can expand this region by calculating the weighted Evans function at the bounds of this region:

\begin{figure}[H]
    \centering
    \includegraphics[width=.7\linewidth]{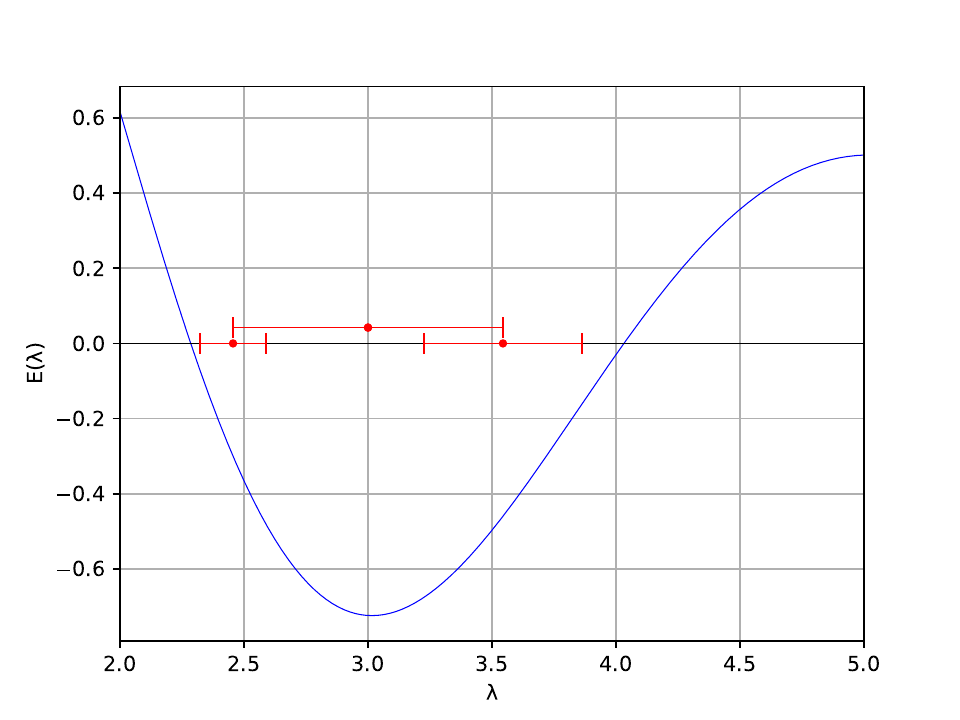}
    \caption{\footnotesize Evans function on $[2,5]$ with weighted Evans function calculated to two iterations at $\lambda = 3$.}
    \label{fig:Fig3}
\end{figure}

According to theorem \ref{thm:QNewton} the quasi-Newton iterations 
\[
\lambda^+_{i+1} = \lambda^+_i + \left\vert\frac{E(\lambda^+_i)}{W(\lambda^+_i)}\right\vert \qquad \qquad \qquad \lambda^-_{i+1} = \lambda^-_i - \left\vert\frac{E(\lambda^-_i)}{W(\lambda^-_i)}\right|
\]
will give sequences $\{\lambda_i^\pm\}_{i=0}^\infty$ which converge to the smallest eigenvalue greater that $\lambda_0$ and the largest eigenvalue less than $\lambda_0$ respectively.  

\subsection{4th-order operator with Dirichlet boundary conditions}

We demonstrate that the process retains its effectiveness in higher dimensions. Consider the operator $\mathcal{L} = \partial_{xxxx} + \partial_x(\cos(x)\partial_x)$ with the boundary conditions $u(0) = u_{xx}(0) = 0$ and $u(L) = u_{xx}(L) = 0$. Let $L = 2\pi$ as before. The Evans function on $[2.5,17.5]$ is:

\begin{figure}[H]
    \centering
    \includegraphics[width=.7\linewidth]{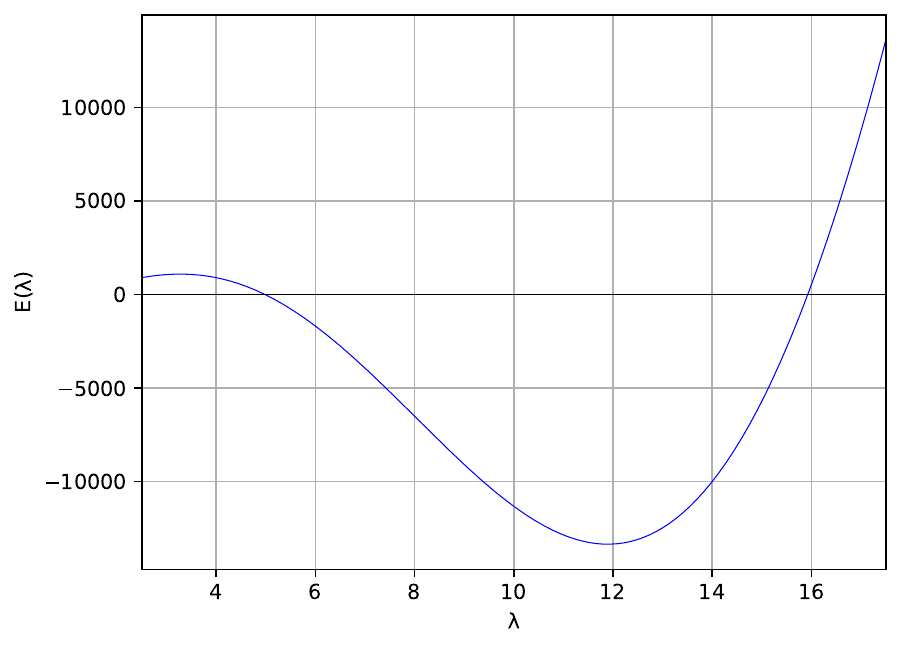}
    \caption{\footnotesize Evans function on $[2.5,17.5]$.}
    \label{fig:Fig4}
\end{figure}

Now we calculate the weighted Evans function about $\lambda = 11$ to three iterations:

\begin{figure}[H]
    \centering
    \includegraphics[width=.7\linewidth]{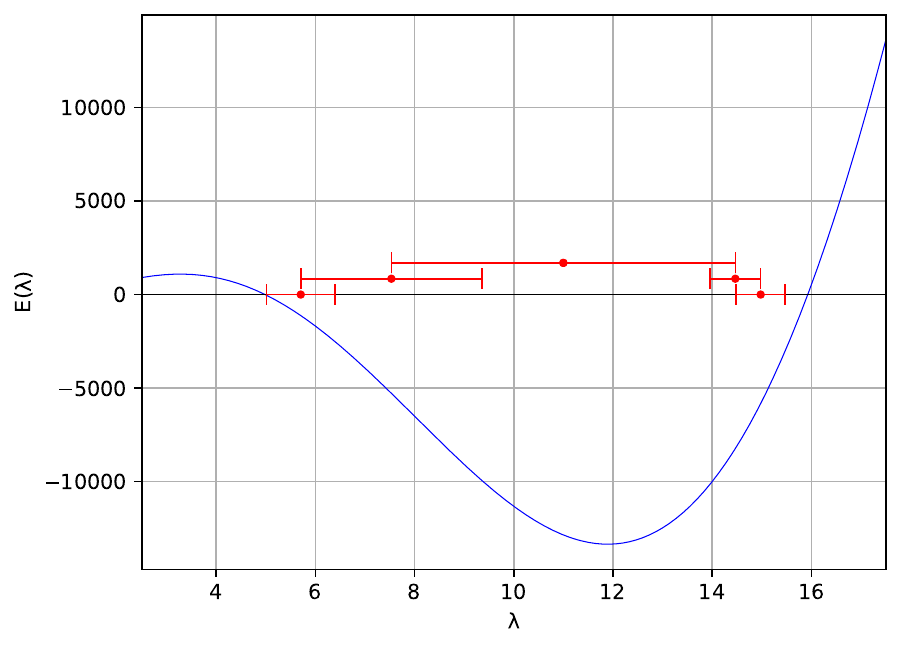}
    \caption{\footnotesize Evans function on $[2.5,17.5]$ with weighted Evans function calculated to three iterations at $\lambda = 11$.}
    \label{fig:Fig5}
\end{figure}

Since the proof of \ref{thm:derivative} assumes a second order operator it does not extend to this case. However, we believe that the theorem may still apply to higher order cases. In figure \ref{fig:Fig6} we plot the weighted Evans function near the eigenvalue at $\approx 5$.

\begin{figure}[H]
    \centering
    \includegraphics[width=.7\linewidth]{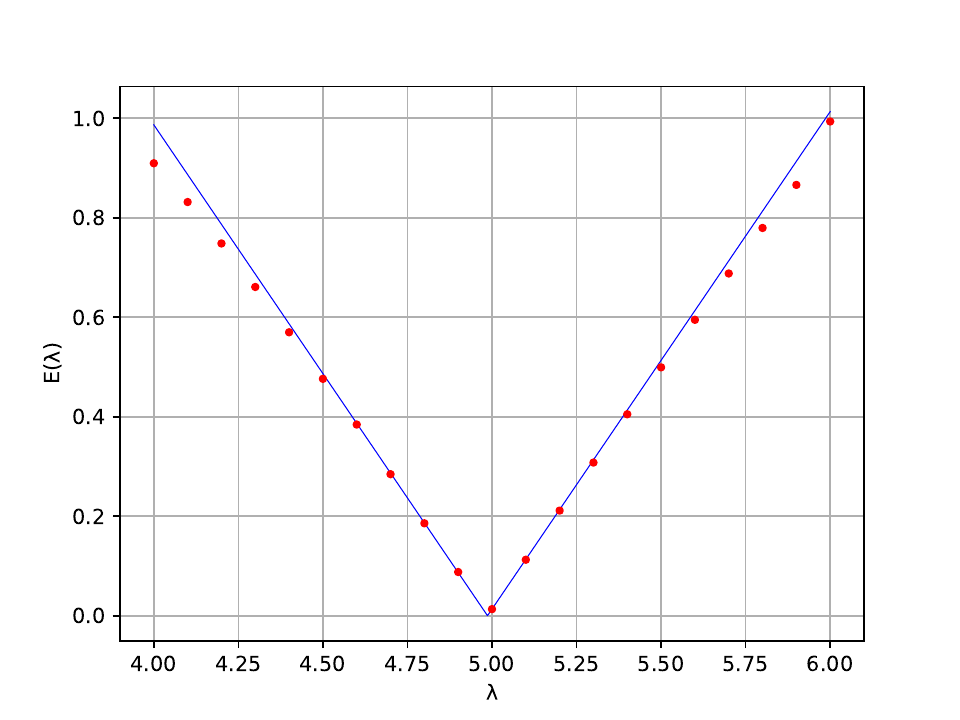}
    \caption{\footnotesize The weighted Evans function plotted at intervals of $.1$ against spectral distance.}
    \label{fig:Fig6}
\end{figure}

\subsection{The linearized mKdV equation}

Now we use this method to study the stability of traveling waves to the generalized KdV equation. For the mKdV equation, the linearized stability of the standing waves is governed by the spectrum of the third order non-self-adjoint operator 
$$\mathcal{L}u = -\partial_{xxx}u + \partial_x(-3\phi^2u)$$
with $\phi$ an elliptic function of $x$. For our example, we choose the Jacobi elliptic function $\textrm{cn}(x;m=\frac{1}{2})$ for the parameter value $m= 1/2$. We can calculate the spectrum via the Floquet-Fourier-Hill method\cite{FFHM} in figure \ref{fig:Fig7}.

\begin{figure}
    \centering
    \includegraphics[width=.7\linewidth]{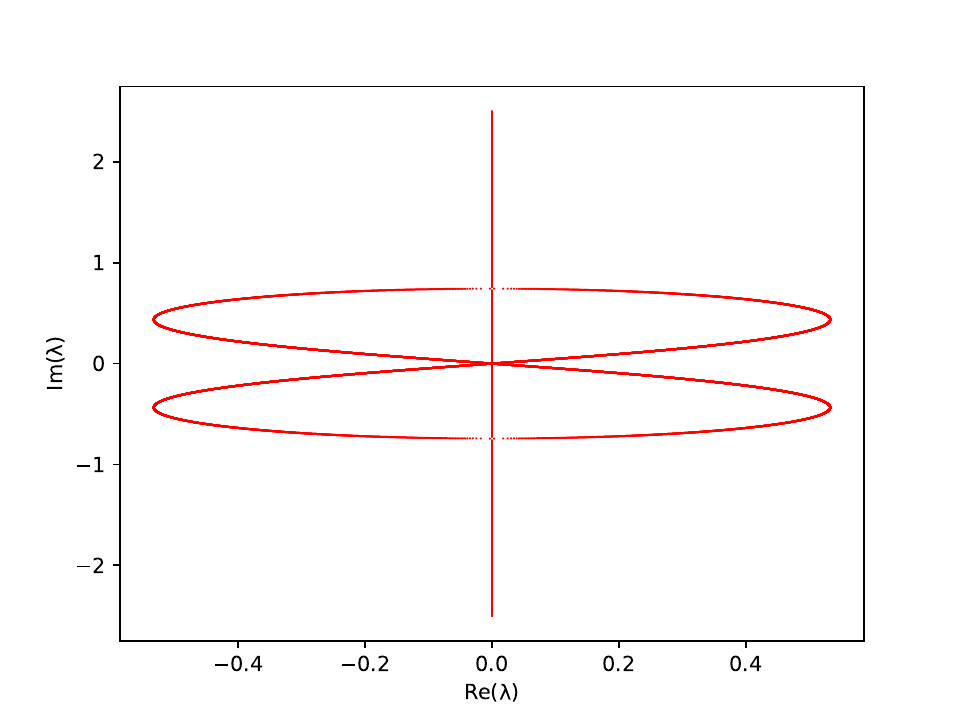}
    \caption{\footnotesize Spectrum of $\mathcal{L}$ in $[-0.6-2.5i, 0.6 + 2.5i]$.}
    \label{fig:Fig7}
\end{figure}

Since $\mathcal{L}$ is not self-adjoint, the spectrum can, in principle, lie anywhere in the complex plane. To calculate the weighted Evans function, we reduce the bounded boundary conditions to periodic ones using Floquet's theorem as described in \ref{subsec:extensionofparams}. We choose $\mu \in [0,2\pi/4K)$ where $K$ is the elliptic integral of the first kind, and consider the associated operator
$$\mathcal{L}(\mu)u \coloneqq -(\partial_x + i\mu)^3u + (\partial_x + i\mu)(-3\phi^2u)$$
with periodic boundary conditions. We will choose $\mu = 0.1$ and $\lambda = 0.1 + 0.5i$.

\begin{figure}
    \centering
    \includegraphics[width=.7\linewidth]{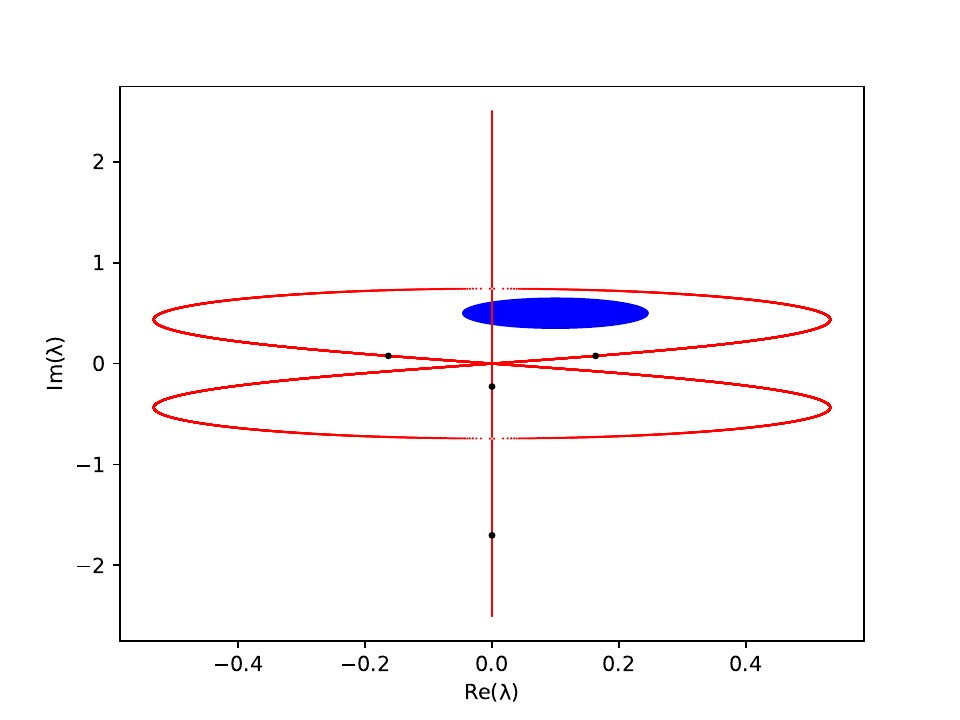}
    \caption{\footnotesize Spectrum of $\mathcal{L}$ in $[-0.6-2.5i, 0.6 + 2.5i]$ with black dots representing spectrum of $\mathcal{L}(\mu)$ and blue circle representing the weighted Evans function at $\lambda = 0.1 + 0.5i$.}
    \label{fig:Fig8}
\end{figure}

Figure \ref{fig:Fig8} represents the essential spectrum $\mathcal{L}$, depicted in red, together with the eigenvalues for $\mu=0.1$ (Black dots). The blue disc represents the guaranteed eigenvalue free region for the Evans function for ${\mathcal L(\mu=0.1)}$ computed at $\lambda = 0.1 + 0.5 i.$ We emphasize that the red curve represents the essential spectrum, which is given by the union of the spectra of $\mathcal{L}(\mu)$ over all $\mu$. The theorem is applied for a fixed $\mu$, so the fact that the blue disc intersects the spectrum does not represent a contradiction: those points represent eigenvalues of $\mathcal{L}(\mu)$ for different $\mu$ values. 

 To recover information about the spectrum of $\mathcal{L}$ we use extension of parameters to compute diamonds in $\lambda$ and $\mu$ in figure \ref{fig:Fig9}.

\begin{figure}
    \centering
    \includegraphics[width=.7\linewidth]{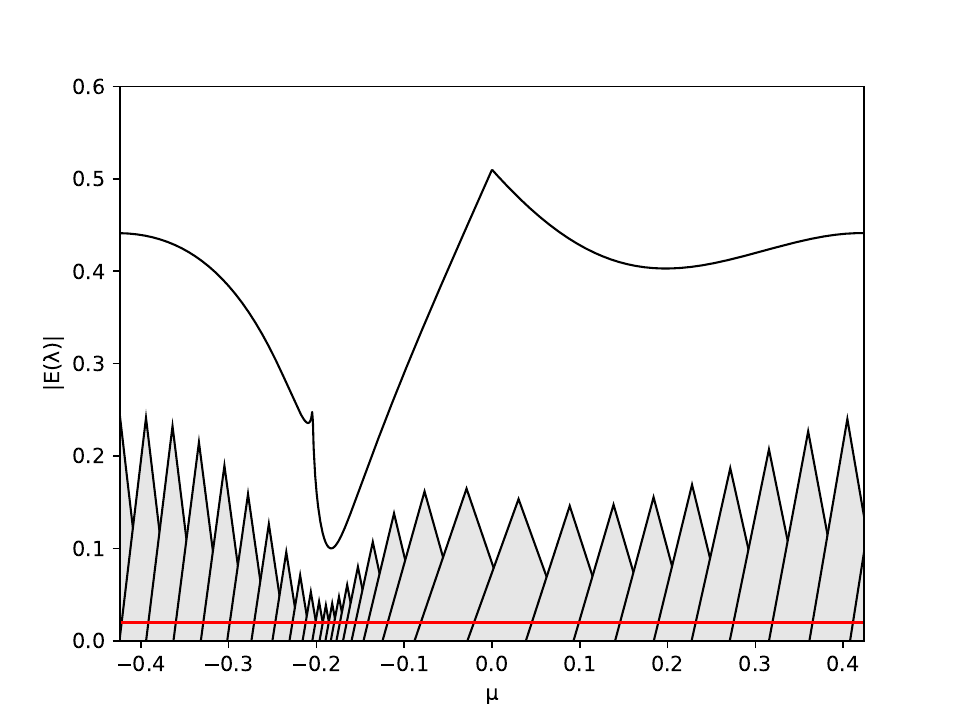}
    \caption{\footnotesize Black line represents spectral distance from $\lambda = 0.1 + 0.5i$ for each value of $\mu$ and grey rectangles represent computations of the weighted Evans function for $\mathcal{L}(\mu)$ at $\lambda = 0.1 + 0.5i$.}
    \label{fig:Fig9}
\end{figure}

By taking the intersection points of the diamonds we calculate the largest possible rectangle that fits in the region covered by the diamonds. The height of this rectangle, $h$, is a lower bound on $\lVert (\mathcal{L} - (0.1 - 0.5i))^{-1}\rVert.$ In the diagram below $h \approx 0.01947$ which produces a region in the spectrum of $\mathcal{L}$.

\begin{figure}
    \centering
    \includegraphics[width=.7\linewidth]{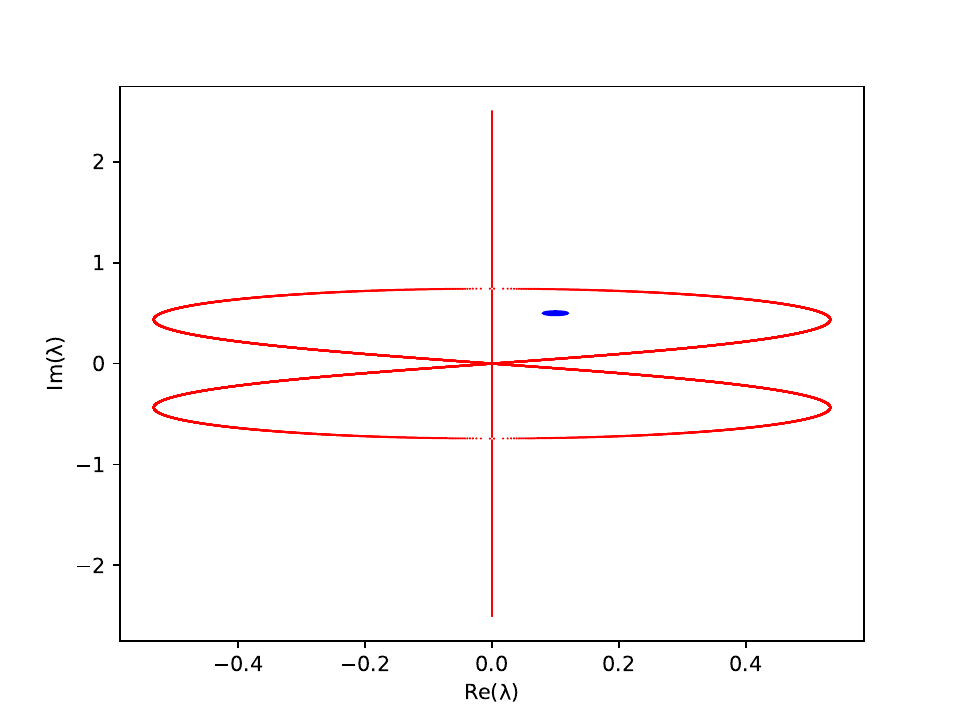}
    \caption{\footnotesize Spectrum of $\mathcal{L}$ in $[-0.6-2.5i, 0.6 + 2.5i]$ with an eigenvalue free region centered at $\lambda = 0.1 + 0.5i$.}
    \label{fig:Fig10}
\end{figure}

\section{Conclusions}

In this paper we have considered the role of the normalization for the Evans function defined by an ODE eigenvalue problem on a compact interval. For a number of problems of the type that arise in practice there is a normalization of the Evans function that is both theoretically simple and numerically easily computable that provides a natural geometric interpretation to the magnitude of the Evans function. In particular the magnitude of the normalized Evans function guarantees a a ball in the spectral or parameter plane that is free of eigenvalues. While the examples here have been given using conventional numerics we see one potentially important application of this as being to rigorous numerics\cite{BJ2022,BBHY2025} In the field of rigorous numerics a common way to show that an interval is free of eigenvalues is by interpolation. One first calculates the Evans function with error bounds at a finite set of points. Given some information on the analyticity properties of a function in a complex region containing the interval one get rigorous bounds on the error in Chebyshev interpolation.
If the error in Chebyshev interpolation is sufficiently small this can be used to the non-existence of roots of the Evans function. This is, of course, a rather involved procedure. The bounds presented here suggest a much simpler alternative. 
\appendix
\section{Proof of Theorem \ref{thm:HS}}
The Hilbert-Schmidt norm of $A$ is 
$$\left\lVert A \right\rVert^{2}_{H,S} = \sum \sigma_{i}$$
where $\sigma_{i}$ are the eigenvalues of $A^{T}A$. Now we note that $A^{T}A$ has the same eigenvalues as $AA^{T}$ which are the reciprocals of the eigenvalues of $(AA^{T})^{-1} = (A^{-1})^{T}A^{-1}.$
So,
$$\left\lVert A^{-1} \right\rVert^{2}_{H,S} = \sum \frac{1}{\sigma_{i}}.$$
Since $\det(A)=1$ we have $\det(AA^{T}) = 1$ meaning, 
$$\prod \sigma_{i} = 1.$$ Evidently,  
$$\prod \frac{1}{\sigma_{i}} = 1.$$
Now consider the set $\{\sigma_{1},\sigma_{2},\dots,\sigma_{d}\}$ ordered from least to greatest. We will do induction on elements of this set. For our base case we consider $\{\sigma_1\}$. We have 
$$\sigma_1^0 \geq \sigma_1 * \frac{1}{\sigma_1}$$
Now for the inductive step choose some $n < d$ and suppose that 
$$\left(\frac{\sum^{n}_{i} \sigma_{i}}{n}\right)^{n-1} \geq \prod_{i}^{n}\sigma_{n}* \frac{\sum_{i}^{n}\frac{1}{\sigma_{i}}}{n}.$$
We wish to show, 
$$\left(\frac{\sigma_{n+1} + \sum^{n}_{i} \sigma_{i}}{n+1}\right)^{n} \geq \sigma_{n+1}\left(\prod_{i}^{n}\sigma_{i}* \frac{\sum_{i}^{n}\frac{1}{\sigma_{i}}}{n+1}\right) + \frac{\prod_{i}^{n}\sigma_{i}}{n+1}.$$
Letting 
$$f(\sigma_{n+1}) \coloneqq \left(\frac{\sigma_{n+1} + \sum^{n}_{i} \sigma_{i}}{n+1}\right)^{n} - \sigma_{n+1}\left(\prod_{i}^{n}\sigma_{i}* \frac{\sum_{i}^{n}\frac{1}{\sigma_{i}}}{n+1}\right)  - \frac{\prod_{i}^{n}\sigma_{i}}{n+1}$$
and evaluating at $\frac{\sum^{n}_{i} \sigma_{i}}{n}$, 
$$f\left(\frac{\sum^{n}_{i} \sigma_{i}}{n}\right) = \left(\frac{\sum^{n}_{i} \sigma_{i}}{n}\right)^{n} - \left(\frac{\sum^{n}_{i} \sigma_{i}}{n}\right)\left(\prod_{i}^{n}\sigma_{i}*\frac{\sum_{i}^{n}\frac{1}{\sigma_{i}}}{n+1}\right) - \frac{\prod_{i}^{n}\sigma_{i}}{n+1}.$$
Now by the AM-HM-inequality,
$$\frac{\sum_{i}^{n}\sigma_{i}}{n}\frac{\sum_{i}^{n}\frac{1}{\sigma_{i}}}{n} \geq 1$$
so, 
$$\hspace{-1.2cm}f\left(\frac{\sum^{n}_{i} \sigma_{i}}{n}\right) \geq \left(\frac{\sum^{n}_{i} \sigma_{i}}{n}\right)^{n} - \left(\frac{\sum^{n}_{i} \sigma_{i}}{n}\right)\left(\prod_{i}^{n}\sigma_{i}*\frac{\sum_{i}^{n}\frac{1}{\sigma_{i}}}{n+1}\right) - \left(\frac{\sum_{i}^{n}\sigma_{i}}{n}\right)\left(\prod_{i}^{n}\sigma_{i}*\frac{\sum_{i}^{n}\frac{1}{\sigma_{i}}}{(n)(n+1)}\right)$$
$$f\left(\frac{\sum^{n}_{i} \sigma_{i}}{n}\right) \geq \left(\frac{\sum^{n}_{i} \sigma_{i}}{n}\right)^{n} - \left(\frac{\sum^{n}_{i} \sigma_{i}}{n}\right)\left(\prod_{i}^{n}\sigma_{i}*\frac{\sum_{i}^{n}\frac{1}{\sigma_{i}}}{n}\right).$$
By the inductive assumption the above is nonnegative. 
Now we take the derivative of $f$; it is, 
$$f'(\sigma_{n+1}) = \frac{n}{n+1}\left[\left(\frac{\sigma_{n+1} + \sum^{n}_{i} \sigma_{i}}{n+1}\right)^{n-1} - \left(\prod_{i}^{n}\sigma_{i}* \frac{\sum_{i}^{n}\frac{1}{\sigma_{i}}}{n}\right)\right].$$
Since $\sigma_{n+1}$ is greater than all the previous $\sigma_{i}$s, it is greater than their arithmetic mean. So, 
$$\frac{\sigma_{n+1} + \sum^{n}_{i} \sigma_{i}}{n+1} \geq \frac{\sum^{n}_{i} \sigma_{i}}{n}.$$
Thus, 
$$f'(\sigma_{n+1}) \geq \frac{n}{n+1}\left[\left(\frac{\sum^{n}_{i} \sigma_{i}}{n}\right)^{n-1} - \left(\prod_{i}^{n}\sigma_{i}* \frac{\sum_{i}^{n}\frac{1}{\sigma_{i}}}{n}\right)\right].$$
By the inductive assumption this is also nonnegative. 
Altogether we have that $f\geq 0$ at $\frac{\sum^{n}_{i} \sigma_{i}}{n}$ and it's derivative is always nonnegative. This implies that evaluated at $\sigma_{n+1} \geq \frac{\sum^{n}_{i} \sigma_{i}}{n}$, we have $f\geq 0$. 
This proves the inductive hypothesis. 
So we find, 
$$\left(\frac{\sum^{d}_{i} \sigma_{i}}{d}\right)^{d-1} \geq \prod_{i}^{d}\sigma_{i}* \frac{\sum_{i}^{d}\frac{1}{\sigma_{i}}}{d}.$$
Substituting the norms, 
$$\left(\frac{\left\lVert A \right\rVert^{2}_{H,S}}{d}\right)^{d-1} \geq \prod_{i}^{d}\sigma_{i}* \frac{\left\lVert A^{-1} \right\rVert^{2}_{H,S}}{d}.$$
Since $\prod_i^{d} \sigma_i = 1$ this gives the desired result,
$$\frac{\left(\left\lVert A \right\rVert^{2}_{H,S}\right)^{d-1}}{d^{d-2}} \geq \left\lVert A^{-1} \right\rVert^{2}_{H,S}.$$

\bibliographystyle{plain} % We choose the "plain" reference style

\bibliography{Paper_Files/EvansFunction} % Entries are in the refs.bib file

\end{document}